%% file: main.tex
\documentclass[preprint, authoryear, 11pt]{elsarticle}

\usepackage[
    letterpaper,
    top=1.0in,
    bottom=1.0in,
    left=1.0in,
    right=1.0in
]{geometry}

\usepackage{setspace}
\usepackage{hyperref}
\usepackage{amsmath,amssymb,amsfonts}
\usepackage{algorithmic}
\usepackage{graphicx}
\usepackage{textcomp}
\usepackage{xcolor}
\usepackage{amsthm}
\usepackage{bm}
\usepackage{dsfont}
\usepackage{subfigure}
\usepackage{subcaption}
\usepackage[algo2e,ruled,noend]{algorithm2e}
\usepackage{mathtools} 
\usepackage{pgfplots}
\usepackage{natbib}

\usepackage{multirow}

\newcommand{\I}{\mathcal I}
\newcommand{\J}{\mathcal J}
\newcommand{\T}{\mathcal T}

\newcommand{\R}{\mathbb R}

\newcommand{\norm}[1]{\left\lVert#1\right\rVert}

\newtheorem{theorem}{Theorem}

\newtheorem{lemma}{Lemma}

\newtheorem{corollary}{Corollary}

\theoremstyle{definition}
\newtheorem{definition}{Definition}

\journal{European Journal of Operational Research}
\begin{document}

\begin{frontmatter}
\title{Pricing Electric Vehicle Charging and Station Access via Copositive Duality}

\author[label1]{Nanfei Jiang\corref{cor1}}
\ead{nanfei@ucsb.edu}

\author[label1]{Yi Zhou}
\ead{yi_zhou@ucsb.edu}

\author[label2]{Josh A. Taylor}
\ead{jat94@njit.edu}

\author[label1]{Mahnoosh Alizadeh}
\ead{mahnoosh@ucsb.edu}

\cortext[cor1]{Corresponding author.}

\affiliation[label1]{
    organization={Department of Electrical and Computer Engineering, 
                  University of California, Santa Barbara},
    city={Santa Barbara},
    state={CA},
    postcode={93106},
    country={USA}
}

\affiliation[label2]{
    organization={Department of Electrical and Computer Engineering,
                  New Jersey Institute of Technology},
    city={Newark},
    state={NJ},
    postcode={07102},
    country={USA}
}

\input{sections/abstract}

\end{frontmatter}

\input{sections/intro}

\input{sections/prob_form}
\input{sections/pricing_mech}

\input{sections/numerical_sim}

\input{sections/conclusion}

\section*{Declaration of Generative AI and AI-assisted technologies in the writing process}

During the preparation of this work, the authors used ChatGPT to improve language and readability. After using this tool, the authors reviewed and edited the content as needed and take full responsibility for the content of the publication.

\bibliographystyle{elsarticle-harv}
\bibliography{references}

\input{sections/appendix}

\end{document}

%% file: sections/abstract.tex
\begin{abstract}
Optimized charging of electric vehicles (EVs) at public locations consists of two decisions: how much energy to deliver at what times, which is continuous, and where to plug in, which is binary. This makes optimizing EV charging a mixed-integer linear program (MILP). This discreteness undermines traditional marginal pricing methods. In this paper, we develop the first marginal-price-based mechanism for pricing EV charging with binary station access constraints. Using the result of \citet{burerCopositiveRepresentationBinary2009}, we express the EV charging as a completely positive program (CPP), whose dual is a copositive program (COP). This convex dual admits valid shadow prices even though the original allocation problem is discrete and nonconvex. By interpreting the COP dual variables as marginal prices, we construct a pricing mechanism that captures EV supply equipment (EVSE) congestion as well as charging-capacity limits. We prove that the resulting mechanism is revenue-adequate for the operator and individually rational for every EV user, in the strong sense that each user maximizes their own welfare by accepting their assigned charging plan rather than deviating to any alternative option. We further develop problem-specific inner-approximation and dimension-reduction techniques that substantially improve the computational tractability of solving the COP in our setting. Numerical experiments on both small and large scale charging instances demonstrate that our pricing mechanism captures discrete congestion effects and aligns user incentives with the system-optimal assignment, outperforming time-of-use (TOU) and convex relaxation benchmarks. 
\end{abstract}

\begin{keyword}
electric vehicle charging \sep copositive programming \sep pricing mechanism
\end{keyword}


%% file: sections/intro.tex
\section{Introduction}
Public EV charging stations are constrained in the amount of power they can provide and the number of vehicles that can charge at once. The latter requires integer constraints because the number of vehicles at a given charger must be integer-valued. In resource allocation problems with convex costs and constraints, one can use the dual variables to design marginal pricing mechanisms with good properties, e.g., rationality and budget balance. However, the need for integer decision variables in EV charging and station access problems undermines standard marginal pricing approaches.

In this paper, we design a marginal pricing mechanism that remains valid in the presence of binary station-access decisions. We use the result of \citet{burerCopositiveRepresentationBinary2009}, which showed that any mixed--binary quadratic program (MBQP) can be equivalently reformulated as a completely positive program (CPP) whose dual is a copositive program (COP). Although solving both CPP and COP is NP-hard, they are convex, and therefore have strong duality if certain regularity conditions are satisfied. We clarify that this reformulation is not intended to make EV charging a computationally easier problem. Rather, it enables us to use convex duality to design pricing mechanisms for discrete station access decisions.

Specifically, we first formulate the EV charging problem as a MILP problem. Then, we construct a pricing mechanism using the variables in the dual COP. The resulting pricing mechanism explicitly captures congestion arising from EVSE scarcity and the facility’s charging-capacity limits. We further prove that the proposed pricing mechanism is (i) revenue-adequate for the system operator, in the sense that the operator’s total net profit is always non-negative, and (ii) individually rational for every EV user, in the sense that under the posted prices, each user maximizes their own welfare by selecting exactly the same EVSE spot given by the system-optimal assignment. We then present strategies for efficiently solving the resulting COP problem, including a variable-reduction technique and an SDP relaxation technique, which enable our method to handle much larger problem instances in practice. 
In the numerical experiments, we compare our pricing mechanisms with standard TOU pricing as well as convex relaxation–based pricing. The results show that, in our experiments, the proposed mechanism outperforms both benchmarks by achieving exact alignment between individual incentives and the system-optimal assignment, while remaining revenue-adequate for the operator.

\noindent\textbf{Related Work:} We review two relevant literature streams: (i) pricing mechanisms for EV charging station access, and (ii) COP and pricing.

\textbf{Pricing mechanisms:} The problem of pricing access to EV charging infrastructure has received growing attention in recent years, with various mechanisms proposed to align user incentives, manage congestion, and minimize grid peaks \citep{valogianni2020sustainable, zaidi2024minimizing, chen2024exponential}. One simple and widely deployed approach is TOU or flat energy pricing, in which users pay a pre-specified price that varies through time. TOU pricing has long been used in electricity tariff design and has also been adopted for EV charging to encourage off-peak charging. Recent representative works include \citep{vuelvas2021time, jones2022impact}. However, TOU pricing does not directly account for the scarcity of individual charging ports or local hosting constraints. As a result, it cannot fully reflect congestion caused by limited EVSE capacity or systematically recover other non-energy costs such as station construction and maintenance.

Motivated by these shortcomings, several pricing mechanisms have been developed that explicitly model congestion and finite EVSE capacity. One line of work constructs prices from the duals of convex relaxations of the underlying MILP charging problem, e.g., \citep{cui2021optimal, tucker2019online}. Related studies within this framework further examine pricing and coordination in coupled power–transportation networks, capturing system-level interactions among charging infrastructures \citep{wang2018coordinated, li2022strategic, alizadeh2016optimal}. These prices are easy to compute, but correspond to non-integral solutions, which can only be interpreted for aggregations of EVs. A second line of work adopts auction-based mechanisms for allocating charging resources \citep{vohra2011mechanism}. However, contrary to marginal pricing, VCG-type pricing generally cannot be posted in advance for users to make their own decisions and requires one counterfactual optimization per user to compute payments.
A third line of work uses queueing theory, where delays arise from stochastic arrivals and finite service capacity, e.g., \citep{moghaddam2019coordinated, zhang2018optimal, lai2022pricing}. While analytically elegant in classical queueing settings, these models do not retain their tractability once practical smart charging algorithms that are optimized based on variable grid costs are employed.



Here we consider the offline setting. However, practical implementations must be online, where EVs arrive sequentially and decisions must be made without knowledge of future demand. For example, convex relaxation-based methods can be adapted to online operation via a real-time primal-dual update, e.g., \citep{buchbinder2009design, chen2022primal}. Auction-based methods can also be made online by clearing the market at pre-defined intervals as users arrive, e.g., \citep{rigas2022mechanism, gerdingOnlineMechanismDesign2011, robu2013online}. Given the practical relevance, we intend to develop online extensions of our pricing mechanisms in future work.

\textbf{Copositive programming and pricing:} \citet{burerCopositiveRepresentationBinary2009} established that, under mild regularity conditions, any MBQP can be reformulated equivalently as a CPP. Since the dual of a CPP is a COP, this result builds a bridge between discrete optimization problems and the COP framework. Leveraging this connection, a rich set of theoretical and algorithmic tools developed for CPP/COP can be brought to bear on the analysis of the original MILP problem, including hierarchical approximations and structural results on copositive cones \citep{parrilo2000structured, bomze2000copositive, dur2010copositive}. 

The approach we adopt builds on the framework introduced in \citet{guoCopositiveDualityDiscrete2025} for electricity market applications. In that work, the MILP for unit commitment was expressed as an equivalent CPP, the dual of which was used to design pricing mechanisms for the on/off decisions of generators.

The rest of this paper is organized as follows. 
Section~\ref{sec: prob_form} formulates the EV charging problem as a MILP problem and constructs its CPP representation and COP dual. Section~\ref{sec: mechanism} uses the COP to formulate a pricing mechanism. Section~\ref{sec: main result} establishes its key theoretical properties, including revenue adequacy and individual rationality. Section~\ref{sec: approx} introduces two methods for improving the computational tractability of the COP in our setting. Section~\ref{sec: numerical} presents numerical simulations to evaluate the proposed pricing mechanisms, including comparisons with TOU pricing and convex relaxation-based pricing.

\section*{Notation}
\begin{align}
    & \mathcal{I} && \text{Set of EVs, indexed by }i, \nonumber \ 
    |\mathcal{I}|=I \\
    & \mathcal{J} && \text{Set of EVSEs, indexed by }j, \ \ 
    |\mathcal{J}|=J \nonumber \\
    & \mathcal{T} && \text{Set of time intervals, indexed by } t, \ \ 
    |\mathcal{T}|=T \nonumber \\
    & v_i  && \text{User $i$'s valuation for receiving charging service} \nonumber \\
    & e_i && \text{EV $i$'s total energy request} \nonumber \\
    & t_i^-,t_i^+ && \text{EV $i$'s arrival and departure time} \nonumber \\
    & s_i(t) && \text{Binary constant for EV $i$ parking at time $t$} \nonumber \\
    & p_j(t) && \text{Energy allocated at EVSE $j$ at time $t$} \nonumber \\
    & \pi(t) && \text{Unit energy cost (grid price) at time $t$} \nonumber \\
    & G_j(t) && \text{EVSE $j$'s max energy output at time $t$} \nonumber\\
    & x_{ij} && \text{Binary assignment for EV $i$ at EVSE $j$} \nonumber \\
    & \phi_{ij}\ge 0 && \text{Energy-delivery slack for EV $i$ and EVSE $j$} \nonumber \\
    & \psi_{jt}\ge 0 && \text{Station-occupancy slack for EVSE $j$ at time $t$} \nonumber \\
    & \zeta_{jt}\ge 0 && \text{Headroom-to-capacity slack for EVSE $j$ at time $t$} \nonumber \\
    & \xi_i\ge 0 && \text{Unassigned-vehicle slack for EV $i$} \nonumber 
\end{align}

Unless otherwise specified, $\norm{\cdot}$ refers to the $L_2$ norm. The notation $[[1,n]]$ represents the set of integers from $1$ to $n$. We use $\bm{1}_n$ to denote an all-one vector with length of $n$. We denote the stacked vector form of variables using bold symbols; for example, the assignment variables $x_{ij}$ are stacked into $\bm{x}^\top = [x_{11}, x_{12}, \dots, x_{ij}, \dots, x_{IJ}]^\top$. We use $\bm{x}_i$ to stand for the sub-vector of $\bm{x}$ that corresponds to user $i$, i.e. $\bm{x}_i^\top = [x_{i1}, \dots, x_{ij}, \dots, x_{iJ}]^\top$. We also use $\bm{x}_{-i}$ to denote the assignments of all other users except user $i$.

%% file: sections/prob_form.tex
\section{Problem Formulation} \label{sec: prob_form}


We consider a charging station with $I$ EVs arriving to use one of $J$ charging spots equipped with EVSEs, indexed by $j \in \J$. Time is discretized into $T$ periods, indexed by $t \in \T$. Each EV user $i$ is characterized by a tuple $(v_i, t_i^{-}, t_i^{+}, e_i)$, where $t_i^{-}$ and $t_i^{+}$ denote the arrival and departure times, $e_i$ is the total energy requirement, and $v_i$ represents the valuation for user $i$ receiving charging service. If EV $i$ is assigned to an EVSE and successfully obtains $e_i$ units of energy, it receives value $v_i$; otherwise, if it is not assigned to any charger, it departs unserved and obtains zero value. For brevity of notation, we assume each EVSE can only serve one EV at a time.

We consider an \emph{offline} setting in which all EV attributes $(v_i, t_i^{-}, t_i^{+}, e_i)$ are known in advance. The system operator solves the following welfare-maximization problem before implementing any charging decisions:
\begin{subequations}\label{eq: MILP}
\begin{align}
\max_{\{x_{ij}\},\,\{p_j(t)\}} & \; 
\sum_{i\in\I}\sum_{j\in\J} v_i x_{ij}
 - \sum_{t\in\T}\pi(t)\sum_{j\in\J}p_j(t) \label{eq: MILP-obj} \\
\text{s.t. } &
\sum_{t\in\T} s_i(t)p_j(t) \ge e_i x_{ij},
\quad \forall i\in\I,\, j\in\J  \label{eq: MILP-energy}\\
& \sum_{i\in\I} s_i(t)x_{ij} \le 1,
\quad \forall j\in\J,\, t\in\T  \label{eq: MILP-occupancy}\\
& 0 \le p_j(t) \le G_j(t),
\quad \forall j\in\J,\, t\in\T  \label{eq: MILP-capacity}\\
& \sum_{j\in\J} x_{ij} \le 1,
\quad \forall i\in\I  \label{eq: MILP-unique}\\
& x_{ij}\in\{0,1\},
\quad \forall i\in\I, j\in \J. \label{eq: MILP-binary}
\end{align}
\end{subequations}

The objective \eqref{eq: MILP-obj} is the sum of the valuations $v_i$ of all EVs that receive service minus the total electricity cost $\pi(t)$. 
Constraint \eqref{eq: MILP-energy} ensures that any EV assigned to a station receives at least its required energy $e_i$ within its availability window. Constraint \eqref{eq: MILP-occupancy} enforces that at most one EV may occupy a given station at any time slot. 
Constraint \eqref{eq: MILP-capacity} enforces the per-time-slot charging power limit $G_j(t)$ at each EVSE. Constraint \eqref{eq: MILP-unique} ensures that each EV is assigned to at most one charging station, and constraint \eqref{eq: MILP-binary} requires the assignment variables $x_{ij}$ to take binary values.

\subsection{Standard MILP Formulation}

To write the MILP as a CPP, we first convert \eqref{eq: MILP} to its standard form. Let $\phi_{ij}$, $\psi_{jt}$, $\zeta_{jt}$, and $\xi_i$ be slack variables corresponding to the inequality constraints \eqref{eq: MILP-energy}-\eqref{eq: MILP-unique}. To streamline notation, let $\bm{z}^\top = \big[\, \bm{x}^\top,\, \bm{p}^\top,\, \bm{\phi}^\top,\, \bm{\psi}^\top,\, \bm{\zeta}^\top,\, \bm{\xi}^\top \,\big] \in \R^N$. The MILP \eqref{eq: MILP} can be written as:

\begin{subequations}\label{eq: std-MILP}
\begin{align}
\text{(MILP)}&\quad \max_{\bm{z}} \quad 
 \sum_{i \in \I}\sum_{j \in \J} v_i x_{ij}
  - \sum_{t \in \T} \pi(t) \!\!\sum_{j \in \J} p_j(t) \label{eq: std-MILP-obj}\\
\text{s.t.} \quad
& \sum_{t \in \T} s_i(t)\,p_j(t) - e_i x_{ij} = \phi_{ij},
\quad \forall i \in \I, j \in \J \label{eq: std-MILP-energy}\\
& 1 - \sum_{i \in \I} s_i(t)\,x_{ij} = \psi_{jt},
\quad \forall j \in \J, t \in \T \label{eq: std-MILP-occupancy}\\
& G_j(t) - p_j(t) = \zeta_{jt},
\quad \forall j \in \J, t \in \T \label{eq: std-MILP-capacity}\\
& 1 - \sum_{j \in \J} x_{ij} = \xi_i,
\quad \forall i \in \I \label{eq: std-MILP-unique}\\
& x_{ij} \in \{0,1\},
\quad \forall i \in \I, j \in \J \label{eq: std-MILP-binary}\\
& \bm{z} \ge \bm{0}. \label{eq: std-MILP-nonneg}
\end{align}
\end{subequations}

\subsection{CPP representation}
\citet{burerCopositiveRepresentationBinary2009} established that the MILP \eqref{eq: std-MILP} can be reformulated as a CPP. The cone of completely positive matrices is:
$$
\mathcal{C}_n^*
:= \left\{X: \ X=\sum_m \bm{w}_m \bm{w}_m^\top,\ \bm{w}_m \in \mathbb{R}_+^n, \ m\ge1 \right\}.
$$

To represent  \eqref{eq: std-MILP} as a CPP, we introduce the variable $Z\in \mathcal{C}_N^*$. Each element of $Z$ corresponds to the product of two elements of $z$. For arbitrary variable collection $\bm{u}$, Let $Z(\bm{u})$ denote the principal submatrix of $Z$.  
For example, when $\bm{u} = (x_{ij}, p_j(t))$, we have
\[
Z(x_{ij}, p_j(t)) 
=
\begin{bmatrix}
Z_{x_{ij},x_{ij}} & Z_{x_{ij},p_j(t)} \\
Z_{p_j(t),x_{ij}} & Z_{p_j(t),p_j(t)} 
\end{bmatrix}.
\]


To encode the linear constraints in \eqref{eq: std-MILP} as trace constraints, we define vectors
$h_{ij}^1$, $h_{jt}^2$, $h_{jt}^3$, and $h_i^4$ that collect the coefficients of the primal variables appearing in the corresponding linear constraints of \eqref{eq: std-MILP}:

$$
(h_{ij}^1)^\top = [e_i,\; -s_i(1), \dots,-s_i(T),\; 1]^\top, \quad \forall i,j
$$
$$
(h_{jt}^2)^\top = [s_1(t),\,\dots,\,s_I(t),\; 1]^\top, \quad \forall j,t
$$
$$
(h_{jt}^3)^\top = \bigl(1,\; 1\bigr), \; \forall j,t, \quad \text{and} \quad
(h_i^4)^\top = \bigl(\mathbf{1}_J^\top,\; 1\bigr), \; \forall i.
$$

Then, consider the CPP:
\begin{subequations} \label{eq: CPP}
\begin{align}
\text{(CPP)} \quad 
&\max_{\bm{z},Z}\; 
\sum_{i \in \I}\sum_{j \in \J} v_i x_{ij}
   - \sum_{t \in \T} \pi(t) \sum_{j \in \J} p_j(t) 
\tag{3} \\
\text{s.t.}\quad
& \sum_{t \in \T} s_i(t)\,p_j(t) - e_i x_{ij} = \phi_{ij},\quad \forall i,j
\tag{$\lambda_{ij}^1$}\label{eq: CPP-energy} \\
& 1 - \sum_{i \in \I} s_i(t)\,x_{ij} = \psi_{jt},\quad \forall j,t
\tag{$\lambda_{jt}^2$}\label{eq: CPP-occupancy} \\
& G_j(t) - p_j(t) = \zeta_{jt},\quad \forall j,t
\tag{$\lambda_{jt}^3$}\label{eq: CPP-capacity} \\
& 1 - \sum_{j \in \J} x_{ij} = \xi_i,\quad \forall i
\tag{$\lambda_{i}^4$}\label{eq: CPP-unique} \\
& \mathrm{Tr}\!\left(h_{ij}^1 (h_{ij}^1)^\top
   Z(x_{ij},p_j(t),\phi_{ij})\right) = 0,\quad \forall i,j
\tag{$\Lambda_{ij}^1$}\label{eq: CPP-trace1} \\
& \mathrm{Tr}\!\left(h_{jt}^2 (h_{jt}^2)^\top
   Z(\{x_{ij}\}_{i\in I},\psi_{jt})\right) = 1,\quad \forall j,t
\tag{$\Lambda_{jt}^2$}\label{eq: CPP-trace2} \\
& \mathrm{Tr}\!\left(h_{jt}^3 (h_{jt}^3)^\top
   Z(p_j(t),\zeta_{jt})\right) = G_j(t)^2,\quad \forall j,t
\tag{$\Lambda_{jt}^3$}\label{eq: CPP-trace3} \\
& \mathrm{Tr}\!\left(h_i^4 (h_i^4)^\top
   Z(\{x_{ij}\}_{j \in \J},\xi_i)\right) = 1,\quad \forall i
\tag{$\Lambda_{i}^4$}\label{eq: CPP-trace4} \\
& x_{ij} = Z_{x_{ij},x_{ij}},\quad \forall i,j
\tag{$\delta_{ij}$}\label{eq: CPP-consistency} \\
& 
\begin{bmatrix}
1 & \bm{z}^\top \\
\bm{z} & Z
\end{bmatrix} \in \mathcal{C}_{N+1}^*.
\tag{$\Omega$}\label{eq: CPP-cone}
\end{align}
\end{subequations}

CPP \eqref{eq: CPP} can be viewed as a convexification of MILP \eqref{eq: std-MILP}: any feasible solution $\bm{z}^*$ to \eqref{eq: std-MILP} induces a feasible solution $(\bm{z},\bm{z}^* (\bm{z}^*)^\top)$ to \eqref{eq: CPP}, and for any optimal solution $(\bm{z},Z)$ to \eqref{eq: CPP}, the vector $\bm{z}$ lies in the convex hull of feasible solutions to \eqref{eq: std-MILP}.

Since \eqref{eq: CPP} is a conic linear program over the CPP cone, its conic dual is constrained over the COP cone $C_n$, defined as
\[
\mathcal{C}_n
:= \left\{\Omega \in \mathbb{S}^n:\ \bm{x}^\top \Omega \bm{x} \ge 0,\ \forall\,\bm{x}\in \mathbb{R}_+^n \right\}.
\]

The exact dual problem of \eqref{eq: CPP} is as follows, where
$\bm{\lambda}:=(\{\lambda_{ij}^1\},\{\lambda_{jt}^2\},\{\lambda_{jt}^3\},\{\lambda_i^4\})$,
$\bm{\Lambda}:=(\{\Lambda_{ij}^1\},\{\Lambda_{jt}^2\},\{\Lambda_{jt}^3\},\{\Lambda_i^4\})$,
$\bm{\delta}:=\{\delta_{ij}\}$, and $\Omega$ correspond to the constraints in \eqref{eq: CPP} and are identified by the tags listed next to each constraint.

\begin{equation} \label{eq: COP}
\begin{aligned}
    \text{(COP)}\quad \min_{\bm{\lambda}, \bm{\Lambda}, \bm{\delta}, \Omega} & \sum_{j,t} (\lambda_{jt}^2 + \Lambda_{jt}^2 + G_j(t)\lambda_{jt}^3  + (G_j(t))^2\Lambda_{jt}^3) \\
    &+ \sum_{i} (\lambda_i^4 + \Lambda_i^4) \\
    \text{subject to}& \quad  (\bm{\lambda}, \bm{\Lambda}, \bm{\delta}, \Omega) \in \mathcal{F}^{\text{COP}}.
\end{aligned}
\end{equation}

$\mathcal{F}^{\mathrm{COP}}$, the feasible set, is explicitly provided in \ref{app: F_COP}. It consists of a collection of linear constraints in $(\bm{\lambda}, \bm{\Lambda}, \bm{\delta}, \Omega)$ together with the copositive constraint $\Omega \in \mathcal{C}_{N+1}$.


%% file: sections/pricing_mech.tex
\section{Our Pricing Mechanism} \label{sec: mechanism}
We use the COP dual to construct a pricing mechanism for \eqref{eq: std-MILP}, which we refer to as copositive marginal pricing (CMP)
\begin{definition}[CMP]
\label{def: pricing}
Let $(\bm{\lambda}^*,\bm{\Lambda}^*)$ be an optimal solution of the COP \eqref{eq: COP}, and let $\bm{z}^*$ be an optimal solution of \eqref{eq: std-MILP}. 
Under the CMP mechanism, the payment charged to each EV $i$, denoted by $\rho_i$, is defined as
\[
\rho_i(\bm{x}_i,\bm{x}_{-i}) 
:= \rho_i^{\mathrm{LS,E}} + \rho_i^{\mathrm{LS,I}} + \rho_i^{c}(\bm{x}_i,\bm{x}_{-i}),
\]
where
\begin{align}
\rho_i^{\mathrm{LS,E}}
&:= \sum_{j,t} \pi(t)\, p_j^*(t)\, s_i(t)\, x_{ij}, 
\label{eq: rho-LS-E}
\\[4pt]
\rho_i^{\mathrm{LS,I}}
&:= \sum_{j,t}\big(\lambda_{jt}^{3,*}+G_j(t)\Lambda_{jt}^{3,*}\big)\, p_j^*(t)\, s_i(t)\, x_{ij}, 
\label{eq: rho-LS-I}
\\[4pt]
\rho_i^{c}(\bm{x}_i,\bm{x}_{-i})
&:= \nonumber \\
\sum_{j,t}\Big(&\lambda_{jt}^{2,*}
+ \Lambda_{jt}^{2,*}\big(1 + 2\!\!\sum_{l\neq i} s_l(t)x_{lj} + 2\psi_{jt}^*\big)\Big) \,
s_i(t)\, x_{ij}.
\label{eq: rho-c}
\end{align}

The three components have the following interpretations:
\begin{itemize}
    \item $\rho_i^{\mathrm{LS,E}}$: a payment for energy consumption.
    \item $\rho_i^{\mathrm{LS,I}}$: a shadow-price payment for EVSE charging-capacity limits.
    \item $\rho_i^{c}$: a congestion charge associated with limited EVSE occupancy, which depends not only on the decision variable $\bm{x}_i$, but also on the assignment of other EV users.
\end{itemize}

At the optimal assignment, the charge term $\rho_i^{c}$ simplifies to
\[
\rho_i^{c}(\bm{x}_i^*,\bm{x}_{-i}^*) 
= \sum_{j,t} (\lambda_{jt}^{2,*} + \Lambda_{jt}^{2,*})\, s_i(t)\, x_{ij}^*.
\]

\end{definition}

\section{Main Results} \label{sec: main result}
\subsection{Revenue Adequacy}
\begin{lemma} \label{lemma: positiveness}
For any feasible solution $(\bm{\lambda}, \bm{\Lambda})$ to \eqref{eq: COP}, we have $\bm{\lambda} \ge 0$ and $\bm{\Lambda} \ge 0$. Consequently,
\[
\rho_i^{\mathrm{LS,I}} \ge 0, \quad \rho_i^{c} \ge 0, \quad \forall i.
\]
\end{lemma}

Lemma~\ref{lemma: positiveness} follows from the structural property of the dual feasible set $\mathcal{F}^{\text{COP}}$. It implies that marginal prices are never negative, and therefore that the operator does not subsidize users due to congestion or charging-capacity constraints. 

The next result follows from the structure of the original MILP formulation: if an EVSE is not allocated to any EV in a particular time slot, no charging power should be supplied at that location and time.

\begin{lemma}
\label{lemma: p*}
Let $\bm{z}^*$ be an optimal solution of \eqref{eq: std-MILP}, then we have that
\begin{equation} p_j^*(t) \cdot \left(1 - \sum_{i \in \I} s_i(t)x_{ij}^*\right) = 0, \quad \forall j \in \J, t \in \T \label{eq: p* property}
\end{equation}
\end{lemma}

\begin{theorem}[Revenue Adequacy] \label{thm: revenue-adequacy}
Let $\rho_i$ denote the payment charged to EV $i$ under CMP mechanism defined in Definition~\ref{def: pricing}. Then the total net profit of the EVSE operator is non-negative, i.e.,
\begin{equation}
    \sum_{i \in \I} \rho_i - \sum_{j \in \J,\, t \in \T} \pi(t)\,p^*_j(t) \;\ge\; 0.
\end{equation}
\end{theorem}

\begin{proof}
The proof of Theorem~\ref{thm: revenue-adequacy} follows directly from Lemma~\ref{lemma: positiveness} and Lemma~\ref{lemma: p*}. By Lemma~\ref{lemma: positiveness}, we have
$$
\rho_i = \rho_i^{\mathrm{LS,E}} + \rho_i^{\mathrm{LS,I}} + \rho_i^{c} \;\ge\; \rho_i^{\mathrm{LS,E}} = \sum_{j,t} \pi(t)\,p_j^*(t)\,s_i(t)\,x_{ij}^*,
\qquad \forall i \in \I.
$$
Summing both sides over all $i \in \I$ yields
$$
\begin{aligned}
\sum_{i \in \I} \rho_i
&\ge \sum_{i \in \I}\sum_{j \in \J,\,t \in \T} \pi(t)\,p_j^*(t)\,s_i(t)\,x_{ij}^* \\
&= \sum_{j \in \J,\,t \in \T} \pi(t)\,p_j^*(t)\left(\sum_{i \in \I} s_i(t)\,x_{ij}^*\right).
\end{aligned}
$$

Next, by Lemma~\ref{lemma: p*}, for every $j \in \J$ and $t \in \T$, we have
$$
p_j^*(t)\Big(\sum_{i \in \I} s_i(t)\,x_{ij}^*\Big) = p_j^*(t).
$$

Substituting this into the previous inequality gives
$$
\begin{aligned}
\sum_{i \in \I} \rho_i
&\ge \sum_{j \in \J,\,t \in \T} \pi(t)\,p_j^*(t)\left(\sum_{i \in \I} s_i(t)\,x_{ij}^*\right) \\
&= \sum_{j \in \J,\,t \in \T} \pi(t)\,p_j^*(t),
\end{aligned}
$$
which completes the proof.
\end{proof}
Theorem~\ref{thm: revenue-adequacy} does not assume strong duality or that $(\bm{\lambda}, \bm{\Lambda})$ be optimal. As long as $(\bm{\lambda}, \bm{\Lambda})$ is feasible for \eqref{eq: COP}, CMP is revenue-adequate. 


\subsection{Individual Rationality}
We show that, for each EV user $i$, once the reservation system assigns the charging decision $\bm{x}_i^*$, the user has no incentive to deviate to any other feasible assignment. In other words, the individual welfare achieved by user $i$ under the assigned solution is no worse than that under any alternative feasible assignment.

Let $\bm{z}^*$ denote an optimal solution to problem~\eqref{eq: std-MILP}. We aim to prove that the decision variables associated with user $i$, defined as $(\hat{\bm{z}}_i^*)^\top = [(\bm{x}_i^*)^\top, (\bm{\phi}_i^*)^\top,\, \xi_i^*]$, also constitute an optimal solution to the following individual subproblem:
\begin{subequations}\label{eq: MILP_i}
\begin{align}
(\text{Sub MILP})\quad & \max_{\hat{\bm{z}}_i}  \; 
\sum_{j\in\J} v_i x_{ij}
 - \rho_i(\bm{x}_i, \bm{x}^*_{-i}) \label{eq: MILP-obj_i} \\
\text{s.t. } &
\sum_{t\in\T} s_i(t)p^*_j(t) - e_i x_{ij} = \phi_{ij},
\quad \forall j\in\J  \label{eq: MILP-energy_i}\\
& 1- \sum_{j\in\J} x_{ij} = \xi_i \label{eq: MILP-unique_i}\\
& x_{ij}\in\{0,1\}, \quad \forall j\in \J \label{eq: MILP-binary_i} \\
& \hat{\bm{z}}_i \ge 0
\end{align}
\end{subequations}
where $\hat{\bm{z}}_i^\top := [\bm{x}_i^\top, \bm{\phi}_i^\top, \xi_i]$ is the associated decision variable for user $i$, and $\rho_i(\bm{x}_i,\bm{x}^*_{-i})$ defined in Definition~\ref{def: pricing} is our design of the payment for user $i$, given her choice $x_{ij}$.

We note that our notion of individual rationality is stronger than the looser requirement used in some prior EV-charging papers, e.g., \citep{rigas2022mechanism, tucker2019online}, where individual rationality only requires that each EV user obtain nonnegative welfare (i.e., the user would not reject the service offer). In contrast, our CMP mechanism not only guarantees nonnegative welfare, but ensures that the assigned service $\bm{x}^*_i$ is the welfare-maximizing choice for each user $i$ under the proposed pricing scheme. Our result is summarized as Theorem \ref{thm: individual rational} in the following. The proof is in the appendix.
\begin{theorem}[Individual Rationality]
\label{thm: individual rational}
If strong duality holds for \eqref{eq: CPP}, then the CMP mechanism in Definition~\ref{def: pricing} is individually rational in the following senses:
\begin{enumerate}
    \item Each EV user obtains nonnegative welfare from the assignment.
    \item No EV user has an incentive to choose a different assignment option.
\end{enumerate}
\end{theorem}

Theorem~\ref{thm: individual rational} establishes a strong form of individual rationality that is rarely attainable in convex relaxation-based or primal-dual methods. The reason is that these methods derive prices from fractional models that do not reflect the true discrete structure, and thus may assign prices that exceed user valuations or fail to support the optimal integral solution. Our CMP mechanism instead, ensures that every EV user benefits from accepting the operator’s assignment, so no user is financially penalized by participating in the mechanism. It further guarantees that, each user maximizes their welfare by choosing exactly the assignment prescribed by the system operator. In other words, the welfare-maximizing central assignment is also a Nash equilibrium of the induced user-level choice game. 

\section{Approximation Techniques}
\label{sec: approx}

The cutting plane algorithm is computationally practical when the size of $\Omega$ is less than about 100. For larger problems, approximate approaches are needed to solve  the COP more efficiently. In this section, we introduce two main strategies we adopt to reduce the computational burden of solving \eqref{eq: COP}.

Our first strategy is problem-specific. We observe that in \eqref{eq: MILP_i}, each EV user's only decision is the variable $\hat{\bm{z}}_i$, while the charging power profile $\bm{p}$ and the auxiliary variable $\bm{\zeta}$ are not directly chosen by individual users. Motivated by this observation, we fix $\bm{p}=\bm{p}^*$ and $\bm{\zeta}=\bm{\zeta}^*$ at their optimal values obtained from the original MILP, and rewrite~\eqref{eq: std-MILP} as the following reduced problem:
\begin{subequations}\label{eq: MILP-powerFixed}
\begin{align}
& \max_{\bm{x}, \bm{\phi}, \bm{\psi}, \bm{\xi}} \quad 
 \sum_{i \in \I}\sum_{j \in \J} v_i x_{ij}
  - \sum_{t \in \T} \pi(t) \!\!\sum_{j \in \J} p^*_j(t) \label{eq: MILP-obj-powerFixed}\\
\text{s.t.} \quad
& \sum_{t \in \T} s_i(t)\,p^*_j(t) - e_i x_{ij} = \phi_{ij},
\quad \forall i \in \I, j \in \J \label{eq: MILP-energy-powerFixed}\\
& 1 - \sum_{i \in \I} s_i(t)\,x_{ij} = \psi_{jt},
\quad \forall j \in \J, t \in \T \label{eq: MILP-capacity-powerFixed}\\
& 1 - \sum_{j \in \J} x_{ij} = \xi_i,
\quad \forall i \in \I \label{eq: MILP-unique-powerFixed}\\
& x_{ij} \in \{0,1\},
\quad \forall i \in \I, j \in \J \label{eq: MILP-binary-powerFixed}\\
& \bm{x}, \bm{\phi}, \bm{\psi}, \bm{\xi} \ge \bm{0}. \label{eq: MILP-nonneg-powerFixed}
\end{align}
\end{subequations}

By fixing $(\bm{p},\bm{\zeta})$ at $(\bm{p}^*,\bm{\zeta}^*)$, the resulting MILP involves fewer decision variables. Consequently, the dimension of the matrix $\Omega$ in the corresponding reduced COP problem becomes much smaller, which greatly reduces the computational burden. We refer to the pricing mechanism induced by solving the reduced COP as CMP-R (meaning CMP with problem reduction), with the following definition:

\begin{definition}[CMP-R]
Let $(\bm{\lambda}^*, \bm{\Lambda}^*)$ denote the dual solution obtained from the reduced COP problem associated with \eqref{eq: MILP-powerFixed}. Under the CMP-R mechanism, the payment for each EV $i$ is defined as
\[
\hat{\rho}_i(\bm{x}_i,\bm{x}_{-i}) 
= \rho_i^{\mathrm{LS,E}} + \hat{\rho}_i^{c}(\bm{x}_i,\bm{x}_{-i}), \quad \forall i \in \I,
\]
where $\hat{\rho}_i^{c}(\bm{x}_i,\bm{x}_{-i})$ has the same expression as $\rho_i^{c}(\bm{x}_i,\bm{x}_{-i})$ in \eqref{eq: rho-c}, but with $\lambda_{jt}^{2,*}$ and $\Lambda_{jt}^{2,*}$ replaced by their counterparts obtained from the reduced COP.
\end{definition}

One possible challenge is that CMP-R no longer retains the charging-capacity dual signals $\bm{\lambda}^3$ and $\bm{\Lambda}^3$. Nevertheless, this does not pose a practical concern: once the operator pre-specifies a feasible power profile $\bm{p}^*$, the charging-capacity constraints are automatically satisfied, and congestion arising from equipment power limits no longer needs to be priced explicitly. The dual solution obtained from the reduced COP problem can still induce a well-behaved pricing mechanism. 

\begin{theorem}\label{thm: reducedCOP}
The CMP-R mechanism is revenue-adequate. Moreover, if strong duality holds for the reduced COP problem, then CMP-R is also individually rational in the same sense as in Theorem~\ref{thm: individual rational}.
\end{theorem}

Theorem~\ref{thm: reducedCOP} follows since fixing $(\bm{p},\bm{\zeta})$ preserves feasibility of the reduced problem and the remaining dual multipliers remain nonnegative. The same arguments as in Theorem~\ref{thm: revenue-adequacy} therefore apply. Individual rationality also holds primarily because $(\bm{p},\bm{\zeta})$ are fixed and are not part of the users' assignment decisions. The detailed proof follows exactly the same arguments as in Theorem~\ref{thm: revenue-adequacy} and Theorem~\ref{thm: individual rational}.

Our second strategy for reducing computational burden is to make use of semidefinite approximations based on the inner and outer approximations of copositive cone \citet{parrilo2000structured}. For example, the copositive cone can be inner approximated by the sum of the positive semidefinite cone, $\mathcal{S}^+$, and the nonnegative cone, $\mathcal{N}$ \citep{maxfield1962matrix}. Thus, instead of enforcing the dual variable $\Omega$ to be copositive, we may require 
$$\Omega \in \mathcal{S}^+ + \mathcal{N} \subsetneq \mathcal{C}_n.$$

This technique can be applied to any COP problem. In particular, when the semidefinite inner approximation is applied to the reduced COP associated with \eqref{eq: MILP-powerFixed}, the resulting dual solution still induces a revenue-adequate pricing scheme. We refer to this mechanism as CMP-RS, short for CMP with problem reduction and semidefinite approximation.

\begin{corollary}[Revenue adequacy for CMP-RS]\label{cor: approx}
Let $(\bm{\lambda}^*, \bm{\Lambda}^*)$ be a solution obtained by solving the reduced COP under semidefinite inner approximation. Then we have $\bm{\lambda}^*\ge 0$, and $\bm{\Lambda}^* \ge 0$. Further, let $\hat{\rho}_i$ denote the payment under CMP-RS. Then CMP-RS is revenue-adequate, i.e.,
$$
\sum_{i \in \I} \hat{\rho}_i - \sum_{j \in \J,\, t \in \T} \pi(t)\,p^*_j(t) \;\ge\; 0.
$$
\end{corollary}

Corollary~\ref{cor: approx} follows by verifying that $(\bm{\lambda}^*, \bm{\Lambda}^*)$ remain feasible for~\eqref{eq: COP}. Lemma~\ref{lemma: positiveness} then guarantees that $\bm{\lambda}^* \ge 0$ and $\bm{\Lambda}^* \ge 0$, which implies revenue adequacy.

%% file: sections/numerical_sim.tex
\section{Numerical Simulations} \label{sec: numerical}
We first study a small-scale numerical example and solve the corresponding COP exactly using the cutting-plane approach proposed in \cite{guoCopositiveDualityDiscrete2025}. The results confirm that the CMP mechanism satisfies all the theoretical properties established earlier and outperforms two benchmark pricing mechanisms. We then evaluate the performance of CMP-RS on a much larger-scale numerical example. The results show that CMP-RS achieves good empirical performance while remaining computationally tractable for large problem instances.

\subsection{CMP Mechanism}
First, we evaluate our CMP mechanism on a small EV charging problem. We consider a system with $I = 5$ EV users, $J = 3$ EVSE charging ports, and a planning horizon of $T = 6$ time slots. Each EV is characterized by its valuation $v_i$, energy requirement $e_i$, and arrival/departure time $[t_i^-, t_i^+]$. All EVSEs have identical power capacity, denoted by $C = 10$ units per time slot. The grid electricity cost follows a profile of $[0.1, 0.2, 0.3,0.4,0.5, 0.6]$ dollars per unit over the six time slots. The parameter configuration of each EV user $i$ for this test case is summarized in Table~\ref{tab:ev-params_big}.
\begin{table}[!htbp]
\centering
\caption{EV Parameter Configuration for Test Case}
\label{tab:ev-params_big}
\begin{tabular}{c|c|c|c}
\hline
User ($i$) & Value $v_i$ & $e_i$ & Service Window $[t_i^-, t_i^+]$ \\
\hline
1 & 2 & 3  & [1, 3] \\
2 & 2 & 6  & [1, 3] \\
3 & 5 & 8  & [1, 3] \\
4 & 4 & 16 & [1, 5] \\
5 & 2 & 7  & [1, 5] \\
\hline
\end{tabular}
\end{table}

We compare our CMP mechanism against two benchmark pricing rules:  
(i) TOU pricing \citep{valogianni2020sustainable}, and  (ii) convex relaxation-based pricing (CRP) \citep{cui2021optimal}.  

The TOU pricing rule charges each EV user solely for its electricity consumption, via the energy lump-sum term $\rho_i^{\mathrm{LS,E}}$ defined in~\eqref{eq: rho-LS-E}. This pricing rule ignores EVSE access scarcity and capacity effects, and therefore is 
not expected to induce individually rational user decisions.

In contrast, the CRP mechanism is constructed by solving a convex relaxation of problem~\eqref{eq: MILP}, i.e., replacing the binary constraints $x_{ij} \in \{0,1\}$ on the assignment variables with linear constraints $0\le x_{ij} \le 1$. The relaxed optimization yields dual variables associated with EVSE occupancy constraints \eqref{eq: MILP-occupancy} and charging-capacity constraints. These dual variables are then interpreted as surrogate ``congestion prices'' and used to form per-user access payments. Accordingly, the CRP mechanism is defined by replacing the COP dual variables $(\bm{\lambda}, \bm{\Lambda})$ in Definition~\ref{def: pricing} with the corresponding dual variables of the relaxed LP.

In the following, we compare our CMP mechanism with the two benchmark mechanisms described above. To highlight the differences in user-level payments under each mechanism, Fig.~\ref{fig: payment} presents a bar chart illustrating the payment incurred by each EV user under the CRP, TOU pricing, and CMP mechanism, respectively. 
\begin{figure}[h]
\centering
\small
\begin{tikzpicture}
\begin{axis}[
    width=0.8\linewidth,
    height=0.4\linewidth,
    ybar,
    bar width=15pt,
    enlarge x limits=0.2,
    ymin=-1,
    ymax=6,
    ylabel={Payments},
    xlabel={EV User $i=1,3,4$},
    symbolic x coords={1 ($v_1 = 2$), 3 ($v_3 = 5$), 4 ($v_4 = 4$)},
    xtick=data,
    legend style={font=\footnotesize, at={(0.5,1.15)}, anchor=south, legend columns=3},
    ticklabel style={font=\footnotesize},
    label style={font=\footnotesize},
    nodes near coords,
    nodes near coords align={vertical},
    every node near coord/.append style={font=\scriptsize, yshift=2pt},
]
\addplot coordinates {(1 ($v_1 = 2$),2.3) (3 ($v_3 = 5$),2.8) (4 ($v_4 = 4$),4.2)};
\addplot coordinates {(1 ($v_1 = 2$),0.3) (3 ($v_3 = 5$),0.8) (4 ($v_4 = 4$),2.6)};
\addplot coordinates {(1 ($v_1 = 2$),1.9) (3 ($v_3 = 5$),2.2) (4 ($v_4 = 4$),3.9)};

\legend{CRP, TOU pricing, CMP}
\end{axis}
\end{tikzpicture}
\caption{Comparison of Payments Under Different Pricing Mechanisms (Users 2 and 5 are not served in the optimal assignment and therefore are omitted from the chart.)}
\label{fig: payment}
\end{figure}
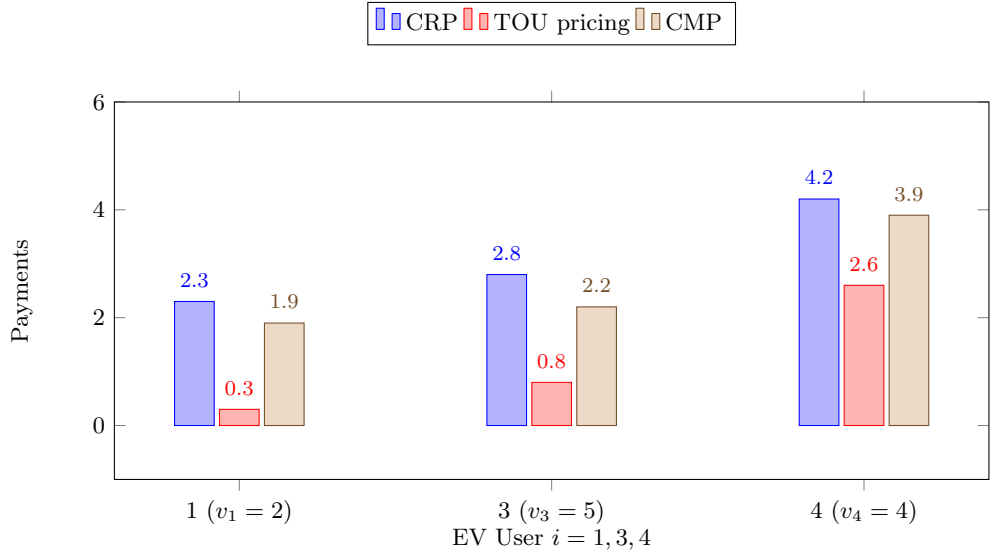

Fig.~\ref{fig: payment} highlights several differences between the three pricing mechanisms. TOU pricing charges users the least, as it accounts only for their electricity consumption and ignores all scarcity signals; it is therefore exactly budget-balanced, since total user payments match the operator's electricity procurement cost. In contrast, the CRP mechanism produces the highest payments, likely because the duals of convex relaxation overestimate the marginal value of binding discrete constraints. As a result, the implied prices may exceed the users’ own valuations~$v_i$, as seen for users~1 and~4 in our example. This will cause rational users to reject their charging offer. Our CMP mechanism, on the other hand, ensures non-negative welfare for every user, which guarantees that they will never reject an offer.

Table~\ref{tab: price-compare} summarizes the assignments of each pricing mechanism. The table compares the central optimal assignment $\bm{z}^*$ with the induced assignment decisions that each user would make by independently solving their own 
welfare-maximization problem under the CRP, TOU pricing, and CMP.
\begin{table}[h]
\centering
\caption{Central Assignment vs.\ User-Preferred Assignments Under Different Pricing Mechanisms}
\label{tab: price-compare}
\begin{tabular}{c|c|c|c|c}
\hline
\textbf{User $i$} 
& \textbf{Central Assignment} 
& \textbf{CRP} 
& \textbf{TOU pricing} 
& \textbf{CMP} \\
\hline
1 & $j=1$ & $j=1$ & $j=1$ & $j=1$ \\
2 & None & $j=2$ & $j=3$ & None \\
3 & $j=2$ & $j=2$ & $j=2$ & $j=2$ \\
4 & $j=3$ & $j=2$ & $j=3$ & $j=3$ \\
5 & None & $j=2$ & $j=3$ & None \\
\hline
\end{tabular}
\end{table}

In Table~\ref{tab: price-compare}, the first column lists the central assignment obtained from 
solving the welfare-maximizing MILP~\eqref{eq: std-MILP}. The notation $j=1$ (or $j=2,3$) indicates that user~$i$ is assigned to EVSE~$j$, i.e., 
$x_{ij}^* = 1$ and $x_{i\ell}^* = 0$ for all $\ell \neq j$; \emph{None} means that 
$x_{ij}^* = 0$ for all $j$, so the user is not served in the assignment. The remaining columns report the induced assignment that each user would choose by independently solving their own welfare-maximization problem. 

The convex relaxation-based pricing fails to reproduce the central assignment because it does not capture the discrete nature of EVSE access. The TOU pricing also fails to align incentives, as it does not use scarcity signals. In contrast, the CMP mechanism induces user decisions that coincide exactly with the central assignment.


\subsection{CMP-RS Mechanism}
Next, we test the performance of the CMP-RS mechanism on a larger-scale numerical example with $I = 25$ EVs, $J = 5$ EVSEs, and $T = 6$ time slots. The resulting dual matrix satisfies $\Omega \in \mathbb{S}^{365}$, indicating a lifting dimension of 365. 
Each EV user’s parameters $\{v_i, e_i, t_i^-, t_i^+\}$ are generated randomly. 
Table~\ref{tab: price-compare_2} reports the user-preferred decisions induced by the three pricing mechanisms under a setting similar to that in Table~\ref{tab: price-compare}.

\begin{table}[h]
\centering
\caption{Central Assignment vs.\ User-Preferred Assignments Under Different Pricing Mechanisms ($I = 25$, $J = 5$)}
\label{tab: price-compare_2}
\begin{tabular}{c|c|c|c|c}
\hline
\textbf{User $i$} 
& \textbf{Central Assignment} 
& \textbf{CRP} 
& \textbf{TOU pricing} 
& \textbf{CMP-RS} \\
\hline
1 & None & None & None & None \\
2 & $j=5$ & $j=1$ & $j=1$ & $j=5$ \\
3 & None & None & $j=1$ & None \\
4 & $j=2$ & $j=2$ & $j=2$ & $j=2$ \\
6 & $j=1$ & $j=1$ & $j=1$ & $j=1$ \\
7 & $j=4$ & $j=1$ & $j=1$ & $j=4$ \\
8 & $j=2$ & $j=2$ & $j=1$ & $j=2$ \\
$\cdots$ & $\cdots$ & $\cdots$ & $\cdots$ & $\cdots$ \\
\hline
\textbf{Matching Rate} 
& --- 
& $52\%$ 
& $40\%$ 
& $\mathbf{100\%}$ \\
\hline
\end{tabular}
\end{table}

As we can see from Table~\ref{tab: price-compare_2}, the CMP-RS mechanism continues to align individual decisions with the central assignment even at a substantially larger scale. Here, the matching rate is defined as the fraction of users whose individually optimal decisions coincide with the central assignment. Both CRP and TOU pricing frequently induce users to deviate from the centrally optimal allocation. Despite relying on approximation techniques, the proposed CMP-RS approach maintains desirable economic properties and demonstrates strong empirical performance, suggesting that the combination of inner-approximation and problem-specific reduction is a practical pathway for scaling the CMP mechanisms to larger instances.


%% file: sections/conclusion.tex
\section{Conclusion} \label{sec: conclusion}
We have used COP to construct CMP, a new pricing mechanism for EV charging station access. Traditional marginal pricing fails in this setting because integer-valued EVSE occupancy makes the allocation problem inherently discrete and non-convex. By reformulating the MILP as a CPP and analyzing its COP dual, we obtain economically meaningful shadow prices that remain valid despite the presence of binary decisions. CMP guarantees revenue adequacy and individual rationality. We demonstrate through numerical experiments that it captures congestion more accurately than convex relaxation-based or fixed-price benchmarks.

Future work includes extending CMP to real-time or stochastic settings (potentially through rolling horizon or primal-dual techniques), improving the scalability of CPP/COP and integrating distribution-network or renewable-generation constraints. 

\section{Acknowledgements}
This work was supported in part by NSF Grant \#2422849 and by the California Energy Commission under Prime Award No.~EPC-25-016 (Demonstrating Electric Vehicle Sub-Metering Solutions, DEVS).

%% file: sections/appendix.tex
\newpage
\onecolumn

\appendix

\section{Helping Lemmas}
\label{app: helping lemmas}
\noindent \textbf{Proof of Lemma \ref{lemma: positiveness}}
\begin{proof}
We prove Lemma \ref{lemma: positiveness} by considering the dual feasible set $\mathcal{F}^{\text{COP}}$ defined in \eqref{eq: F_COP}. Observe that \eqref{eq: F_COP} forces the dual variable $\Omega$ to be a copositive matrix. Thus for all non-negative vectors $\bm{\alpha}$, we have $\bm{\alpha}^\top \Omega \bm{\alpha} \ge 0$. This also holds for any principal submatrix of $\Omega$.

We first prove that $\lambda_{jt}^2, \Lambda_{jt}^2 \ge0$ for all $j,t$. We consider the following two-element index set $i' = [1, i_{\psi_{jt}}]$, where $i_{\psi_{jt}}$ corresponds to the index of $\psi_{jt}$, which is one of the decision variables in \eqref{eq: CPP}. Then the principal submatrix of $\Omega$ under $I'$ is as follows, based on \eqref{eq: F_COP}:
$$ \Omega_{I'} = 
\begin{bmatrix}
0 & \frac{1}{2}\lambda_{jt}^2 \\
\frac{1}{2}\lambda_{jt}^2 & \Lambda_{jt}^2
\end{bmatrix} \in \mathcal{C}_2,
$$
where $\mathcal{C}_2$ denotes the set of all 2-by-2 copositive matrices. It is well-known that under four dimension, any copositive matrix can be decomposed into a non-negative matrix plus a positive semidefinite matrix \citep{maxfield1962matrix}. Thus 
$$ 
\begin{bmatrix}
0 & \frac{1}{2}\lambda_{jt}^2 \\
\frac{1}{2}\lambda_{jt}^2 & \Lambda_{jt}^2
\end{bmatrix} = P + N,
$$
where $P$ denotes a positive semidefinite matrix and $N$ denotes a non-negative matrix. We know that for a positive semidefinite matrix of the form $\begin{bmatrix}
0 & p_{12} \\
p_{21} & p_{22}  
\end{bmatrix}$, 
we must have $p_{22} \ge 0$ and $p_{12} = p_{21}= 0$. Thus we obtain that
$$ 
\begin{bmatrix}
0 & \frac{1}{2}\lambda_{jt}^2 \\
\frac{1}{2}\lambda_{jt}^2 & \Lambda_{jt}^2
\end{bmatrix} = P + N \ge P \ge 
\begin{bmatrix}
0 & 0 \\
0 & 0
\end{bmatrix},
$$
which proves $\lambda_{jt}^2, \Lambda_{jt}^2 \ge 0$, for all $j \in \J,t \in \T$. The exact same argument can also be used to prove the positiveness of $\lambda_{jt}^3, \Lambda_{jt}^3$ and $\lambda_i^4, \Lambda_i^4$. 

\end{proof}

\noindent \textbf{Proof of Lemma~\ref{lemma: p*}}
\begin{proof}
We prove Lemma~\ref{lemma: p*} by examining the value of $\sum_{i \in \I} s_i(t)x_{ij}^*$. If $\sum_{i \in \I} s_i(t)x_{ij}^* = 1$, then $1 - \sum_{i \in \I} s_i(t)x_{ij}^* = 0$, and \eqref{eq: p* property} holds automatically. 

If $\sum_{i \in \I} s_i(t)x_{ij}^* \not = 1$, we know that $x^*_{ij} \in \{0,1\}$, and $\sum_{i \in \I} s_i(t)x_{ij}^* \le 1$. This implies that $\sum_{i \in \I} s_i(t)x_{ij}^* = 0$. In this case, one can verify that if $p_j^*(t)\not= 0$, then we can always replace $p_j^*(t)$ by 0, and this leads to a strictly larger objective value of \eqref{eq: MILP}. This contradicts the fact that $p_j^*(t)$ is optimal. Thus $p_j^*(t)$ can only be 0 in this case. 

Combining the above two cases, we have shown that \eqref{eq: p* property} holds.
\end{proof}

\noindent \textbf{Proof of Theorem \ref{thm: individual rational}}
\begin{proof}
Let $\bm{z}^*$ be an optimal solution of \eqref{eq: std-MILP}. 
Then $\left(\bm{z}^*, \bm{z}^*(\bm{z}^*)^\top \right)$ is an optimal solution of \eqref{eq: CPP}. 
Consider the Lagrangian relaxation obtained by dualizing the constraints corresponding to dual variables $(\lambda_{ij}^1, \Lambda_{ij}^1,\lambda_{jt}^2,\Lambda_{jt}^2,\lambda_{jt}^3,\Lambda_{jt}^3)$:
\begin{subequations}\label{eq: CPP relaxed}
\begin{alignat}{2}
\text{(Relaxed CPP)}\quad 
\max_{\bm{z},\,Z}\;\; 
& \underbrace{\sum_{i\in\I}\sum_{j\in\J} v_i x_{ij}
  - \sum_{t\in\T}\pi(t)\sum_{j\in\J} p_j(t)}_{\text{original objective}}
\label{eq: CPP relaxed-obj} \\
-\sum_{j\in\J}\sum_{t\in\T} \lambda_{jt}^{2,*}
&\Bigl(\sum_{i\in\I} s_i(t)\,x_{ij} + \psi_{jt}-1\Bigr)
-\sum_{j\in\J}\sum_{t\in\T} \Lambda_{jt}^{2,*}
\Bigl(\mathrm{Tr}\!\big(h_{jt}^2 (h_{jt}^2)^\top Z(\{x_{ij}\}_{i \in \I},\psi_{jt})\big) - 1\Bigr)
\notag\\
-\sum_{j\in\J}\sum_{t\in\T} \lambda_{jt}^{3,*}
&\Bigl(p_j(t) + \zeta_{jt} - G_j(t)\Bigr)
-\sum_{j\in\J}\sum_{t\in\T} \Lambda_{jt}^{3,*}
\Bigl(\mathrm{Tr}\!\big(h_{jt}^3 (h_{jt}^3)^\top Z(p_j(t),\zeta_{jt})\big) - G_j(t)^2\Bigr)
\notag\\
\text{s.t. } \quad & \sum_{t \in \T} s_i(t)\,p_j(t) - e_i x_{ij} = \phi_{ij}, \forall i,j\\
& 1 - \sum_{j \in \J} x_{ij} = \xi_i, \quad \forall i \\
& \mathrm{Tr}\!\left(h_{ij}^1 (h_{ij}^1)^\top
   Z(x_{ij},p_j(t),\phi_{ij})\right) = 0,\quad \forall i,j \\
&\mathrm{Tr}\!\left(h_i^4 (h_i^4)^\top
   Z(\{x_{ij}\}_{j \in \J},\xi_i)\right) = 1,\quad \forall i \\
& x_{ij} = Z_{x_{ij},x_{ij}}, \quad \forall i,j \label{eq: CPP relaxed-binary} \\[0.5em]
& \begin{bmatrix}
1 & \bm{z}^\top \\
\bm{z} & Z
\end{bmatrix} \in \mathcal{C}^* \label{eq: CPP relaxed-cone}
\end{alignat}
\end{subequations}

If the Lagrangian multipliers are optimal for \eqref{eq: COP}, then by strong duality, we know that $\left(\bm{z}^*, \bm{z}^*(\bm{z}^*)^\top \right)$ is also an optimal solution of problem \eqref{eq: CPP relaxed}.

Observe that Problem \eqref{eq: CPP relaxed} is equivalent to the following MBQP problem:
\begin{subequations}\label{eq: MBQP}
\begin{alignat}{2}
\text{(MBQP)}\quad 
\max_{\bm{z}}&\;\; 
\sum_{i\in\I}\sum_{j\in\J} v_i x_{ij}
  - \sum_{t\in\T}\pi(t)\sum_{j\in\J} p_j(t)
\label{eq: MBQP relaxed-obj} \\
&
-\sum_{j\in\J}\sum_{t\in\T} \lambda_{jt}^{2,*}
\left(\sum_{i\in\I} s_i(t)\,x_{ij} + \psi_{jt}-1\right)
-\sum_{j\in\J}\sum_{t\in\T} \Lambda_{jt}^{2,*}
\left( \left(\sum_{i\in\I} s_i(t)\,x_{ij} + \psi_{jt}\right)^2-1\right)
\notag\\
&
-\sum_{j\in\J}\sum_{t\in\T} \lambda_{jt}^{3,*}
\left(p_j(t) + \zeta_{jt} - G_j(t)\right)
-\sum_{j\in\J}\sum_{t\in\T} \Lambda_{jt}^{3,*}
\left( \left(p_j(t) + \zeta_{jt} \right)^2- G_j(t) \right)
\notag\\
\text{s.t. } \quad & \sum_{t \in \T} s_i(t)\,p_j(t) - e_i x_{ij} = \phi_{ij}, \forall i \in \I,j\in \J\\
&1 - \sum_{j \in \J} x_{ij} = \xi_i, \quad \forall i \\
& x_{ij} \in \{0,1\},
\quad \forall i \in \I, j \in \J \\
& \bm{z} \ge \bm{0}.
\end{alignat}
\end{subequations}

Notice that the decision variables of the sub MILP problem \eqref{eq: MILP_i} does not contain $\bm{p}, \bm{\psi}, \bm{\zeta}$, which are determined by the EVSE owner. Therefore, in the relaxed MBQP problem, we can fix $(\bm{p}, \bm{\psi}, \bm{\zeta}) = (\bm{p}^*, \bm{\psi}^*, \bm{\zeta}^*)$.

After some manipulation, we can rewrite (A.2a) as
$$\begin{aligned}
\max_{\hat{\bm{z}}} &\quad
\sum_{i\in\I}\sum_{j\in\J} v_i x_{ij}
  - \sum_{i\in \I} \rho_i^{\text{LS,E}} - \sum_{i\in \I} \rho_i^{\text{LS,I}}\\
&\quad
-\sum_{j\in\J}\sum_{t\in\T} \lambda_{jt}^{2,*}
\left(\sum_{i\in\I} s_i(t)\,x_{ij}\right)
-\sum_{j\in\J}\sum_{t\in\T} \Lambda_{jt}^{2,*}
\left(\sum_{i\in\I} s_i(t)\,x_{ij} + \psi^*_{jt}\right)^2,
\notag 
\end{aligned}
$$

By the above equivalence, we know that $\bm{z}^*$ is also an optimal solution of the following problem:
\begin{subequations}\label{eq: MBQP_2}
\begin{alignat}{2}
 \max_{\hat{\bm{z}}_i, \hat{\bm{z}}_{-i}} &\quad J(\bm{x}_i,\bm{x}_{-i} ) := 
\sum_{i\in\I}\sum_{j\in\J} v_i x_{ij}
  - \sum_{i\in \I} \rho_i^{\text{LS,E}} - \sum_{i\in \I} \rho_i^{\text{LS,I}}\\
&\quad
-\sum_{j\in\J}\sum_{t\in\T} \lambda_{jt}^{2,*}
\left(\sum_{i\in\I} s_i(t)\,x_{ij}\right)
-\sum_{j\in\J}\sum_{t\in\T} \Lambda_{jt}^{2,*}
\left(\sum_{i\in\I} s_i(t)\,x_{ij} + \psi^*_{jt}\right)^2, \notag\\
\text{s.t. } \quad & \sum_{t \in \T} s_i(t)\,p^*_j(t) - e_i x_{ij} = \phi_{ij}, \quad \forall i \in \I,j\in \J\\
&1 - \sum_{j \in \J} x_{ij} = \xi_i, \quad \forall i \in \I \\
& x_{ij} \in \{0,1\},
\quad \forall i \in \I, j \in \J \\
& \bm{z} \ge \bm{0},
\end{alignat}
\end{subequations}
where $\hat{\bm{z}}_i := [\bm{x}_i, \bm{\phi}_i, \xi_i]$ denotes all the decision variables that corresponds to user $i$, and $\hat{\bm{z}}_{-i}$ denote the collection of all other decision variables.

Notice that, since $\hat{\bm{z}}^*$ is an optimal solution, we can first maximize problem \eqref{eq: MBQP_2} over $\hat{\bm{z}}_{-i}$ and then maximize over $\hat{\bm{z}}_i$. Thus, by first maximizing over $\hat{\bm{z}}_{-i}$, we obtain that $\hat{\bm{z}}^*_i$ is an optimal solution of the following sub-problem:
\begin{subequations}\label{eq: sub_MBQP}
\begin{alignat}{2}
\quad \max_{\hat{\bm{z}}_i} &\quad J(\bm{x}_i,\bm{x}^*_{-i}) \\
&1 - \sum_{j \in \J} x_{ij} = \xi_i, \\
& x_{ij} \in \{0,1\},
\quad \forall  j \in \J \\
& \hat{\bm{z}}_i \ge \bm{0}.
\end{alignat}
\end{subequations}

Notice that, \eqref{eq: sub_MBQP} has exactly the same constraints as problem \eqref{eq: MILP_i} for all $i$. Further, we have that 
$$\begin{aligned} J(\bm{x}_i,\bm{x}^*_{-i}) &= 
\sum_{j\in\J} v_i x_{ij} - \rho_i^{\text{LS,E}} -\rho_i^{\text{LS,I}} - \sum_{j\in\J}\sum_{t\in\T} \lambda_{jt}^{2,*} s_i(t)\,x_{ij}\\
&\quad -\sum_{j,t}\Lambda_{jt}^{2,*}
\left(s_i(t)x_{ij}\left(s_i(t)x_{ij} + 2\sum_{l\not=i}s_l(t)x_{lj}^* + 2\psi_{jt}^* \right)\right) - \sum_{j,t}\Lambda_{jt}^{2,*} (\psi_{jt}^*)^2
\\
&= \sum_{j\in\J} v_i x_{ij}
 - \rho_i(\{x_{ij}\}) - \sum_{j,t}\Lambda_{jt}^{2,*} (\psi_{jt}^*)^2,
\end{aligned}
$$
which is exactly the objective of \eqref{eq: MILP_i} plus a constant term, where “constant” means independent of $\bm{x}_i$. Therefore, we conclude that \eqref{eq: MILP_i} and \eqref{eq: sub_MBQP} are equivalent, and thus $\hat{\bm{z}}^*$ is an optimal solution of \eqref{eq: MILP_i}. This establishes claim~2) of Theorem~\ref{thm: individual rational}, namely that no user $i$ has incentive to deviate to an alternative assignment.

Next, we prove claim~1) by showing that $\sum_{j\in\J} v_i x_{ij}^* - \rho_i(\{x_{ij}^*\}) \ge 0$ for every $i$. We proceed by contradiction. Suppose that there exists some user $i$ for which
$\sum_{j\in\J} v_i x_{ij}^* - \rho_i(\bm{x}^*_{i}, \bm{x}^*_{-i}) < 0$. Consider now the optimality condition of \eqref{eq: MBQP}. If we replace every $x_{ij}^*$ by $0$ and set $p_j^*(t) = 0$ for all $(j,t)$ satisfying $s_i(t)x_{ij}^* \neq 0$, $j\in\J$, $t\in\T$, then the objective value of \eqref{eq: MBQP relaxed-obj} would strictly increase. This contradicts the assumed optimality of $\bm{z}^*$ for \eqref{eq: MBQP}.

In conclusion, by establishing both claim~1) and claim~2), we complete the proof of Theorem~\ref{thm: individual rational}.

\end{proof}

\section{Formulation of $\mathcal{F}^{\text{COP}}$}
\label{app: F_COP}
To derive the dual feasible set $\mathcal{F}^{\text{COP}}$, we first rewrite the CPP \eqref{eq: CPP} in a more compact form. Let
\[
Y := \begin{bmatrix}
1 & \bm{z}^\top \\
\bm{z} & Z
\end{bmatrix}, \qquad 
\bm{z}^\top = \bigl[\bm{x}^\top,\; \bm{p}^\top,\; \bm{\phi}^\top,\; \bm{\psi}^\top,\; \bm{\zeta}^\top,\; \bm{\xi}^\top \bigr] \in \R^N,
\]
with $Z \in \R^{N\times N}$ and $N$ denoting the length of $\bm{z}$. The CPP can be rewritten as
\begin{subequations}
\label{eq: CPP-rewrite}
\begin{alignat}{4}
\text{(CPP$'$)} \quad \min_{\bm{z},\,Z} \quad & \mathrm{Tr}(\tilde{Q}Y) \\
\text{s.t.}\quad 
& \mathrm{Tr}(H_{ij}^1 Y) = 0, && \forall i\in\I, j\in\J &&& (\lambda_{ij}^1) \\
& \mathrm{Tr}(H_{jt}^2 Y) = 1, && \forall j\in\J, t\in\T &&& (\lambda_{jt}^2) \\
& \mathrm{Tr}(H_{jt}^3 Y) = G_j(t), && \forall j\in\J, t\in\T &&& (\lambda_{jt}^3) \\
& \mathrm{Tr}(H_i^4 Y) = 1, && \forall i\in\I &&& (\lambda_i^4) \\
& \mathrm{Tr}(\tilde{H}_{ij}^1 Y) = 0, && \forall i\in\I, j\in\J &&& (\Lambda_{ij}^1) \\
& \mathrm{Tr}(\tilde{H}_{jt}^2 Y) = 1, && \forall j\in\J, t\in\T &&& (\Lambda_{jt}^2) \\
& \mathrm{Tr}(\tilde{H}_{jt}^3 Y) = G_j(t)^2, && \forall j\in\J, t\in\T &&& (\Lambda_{jt}^3) \\
& \mathrm{Tr}(\tilde{H}_i^4 Y) = 1, && \forall i\in\I &&& (\Lambda_i^4) \\
& \mathrm{Tr}(B_{ij} Y) = 0, && \forall i\in\I, j\in\J &&& (\delta_{ij}) \\
& Y \in \mathcal{C}^* &&&&& (\Omega).
\end{alignat}
\end{subequations}

Here $\tilde{Q}$, $H$, $\tilde{H}$, and $B$ are coefficient matrices corresponding to squared versions of linear constraints.
Then the dual feasible set is given by
\begin{equation}
\label{eq: F_COP}
\begin{aligned}
\mathcal{F}^{\text{COP}} := \Bigl\{ (\bm{\lambda},\bm{\Lambda},\bm{\delta},\Omega) :\;& \mathrm{Tr} \left(\bm{z}^* (\bm{z}^*)^\top \Omega \right) = 0, \quad \tilde{Q}
- \sum_{i\in\I}\sum_{j\in\J}\bigl(\lambda_{ij}^1 H_{ij}^1 + \Lambda_{ij}^1 \tilde{H}_{ij}^1 \bigr) \\
& - \sum_{j\in\J}\sum_{t\in\T}\bigl(\lambda_{jt}^2 H_{jt}^2 + \Lambda_{jt}^2 \tilde{H}_{jt}^2\bigr) 
- \sum_{j\in\J}\sum_{t\in\T}\bigl(\lambda_{jt}^3 H_{jt}^3 + \Lambda_{jt}^3 \tilde{H}_{jt}^3\bigr) \\
& - \sum_{i\in\I}\bigl(\lambda_i^4 H_i^4 + \Lambda_i^4 \tilde{H}_i^4\bigr)
- \sum_{i\in\I}\sum_{j\in\J} \delta_{ij} B_{ij}
+ \Omega = 0,\;\; \Omega \in \mathcal{C}_{N+1} \Bigr\},
\end{aligned}
\end{equation}
where $\mathrm{Tr} \left(\bm{z}^* (\bm{z}^*)^\top \Omega \right) = 0$ is an additional complementary-slackness constraint that we impose to ensure the dual variable associated with $\bm{z}^*$ satisfies the KKT conditions.

Since
\[
\bm{x} \in \R^{IJ}, \quad \bm{p} \in \R^{JT}, \quad \bm{\phi} \in \R^{IJ}, \quad 
\bm{\psi} \in \R^{JT}, \quad \bm{\zeta} \in \R^{JT}, \quad \bm{\xi} \in \R^{I},
\]
we obtain
\[
N = 2IJ + 3JT + I.
\]

We next specify the matrices appearing in \eqref{eq: CPP-rewrite}. The objective matrix is
\[
\tilde{Q}=
\begin{bmatrix}
0 & \tfrac12 q^\top\\
\tfrac12 q & \bm{0}^{N\times N}
\end{bmatrix},
\]
where \(q\in\mathbb{R}^N\) is the coefficient vector of the linear objective in \eqref{eq: CPP}. More precisely, if
\[
\bm{z}^\top=
\bigl[
\bm{x}^\top,\;
\bm{p}^\top,\;
\bm{\phi}^\top,\;
\bm{\psi}^\top,\;
\bm{\zeta}^\top,\;
\bm{\xi}^\top
\bigr],
\]
then
\[
q^\top=
\bigl[
-\bm{v}^\top,\;
\bm{\pi}^\top,\;
\bm{0}^{1\times(IJ+2JT+I)}
\bigr],
\]
where \(\bm{v}\in\mathbb{R}^{IJ}\) and \(\bm{\pi}\in\mathbb{R}^{JT}\) collect the coefficients of \(\bm{x}\) and \(\bm{p}\), respectively. For each \(i\in\I\) and \(j\in\J\), the matrix \(B_{ij}\in\mathbb{S}^{N+1}\) is defined by
\[
(B_{ij})_{l_1,l_2}=
\begin{cases}
\tfrac12,
& l_1=1,\ l_2=1+\mathrm{pos}(x_{ij}),\\[0.3em]
\tfrac12,
& l_1=1+\mathrm{pos}(x_{ij}),\ l_2=1,\\[0.3em]
-1,
& l_1=l_2=1+\mathrm{pos}(x_{ij}),\\[0.3em]
0,
& \text{otherwise},
\end{cases}
\]
where \(\mathrm{pos}(x_{ij})\) denotes the position of \(x_{ij}\) in the vector \(\bm{z}\). Next, consider a linear constraint of the form
\[
\hat{h}^\top \bm{z}=b,
\qquad \hat{h}\in\mathbb{R}^N.
\]

Its associated matrices are defined as
\[
H(\hat{h})=
\begin{bmatrix}
0 & \tfrac12 \hat{h}^\top\\
\tfrac12 \hat{h} & \bm{0}^{N\times N}
\end{bmatrix},
\qquad
\tilde{H}(\hat{h})=
\begin{bmatrix}
0 & \bm{0}^\top\\
\bm{0} & \hat{h}\hat{h}^\top
\end{bmatrix}.
\]

Then
\[
\mathrm{Tr}\bigl(H(\hat{h})Y\bigr)=\hat{h}^\top \bm{z},
\qquad
\mathrm{Tr}\bigl(\tilde{H}(\hat{h})Y\bigr)=\hat{h}^\top Z\hat{h}.
\]

Accordingly, each matrix \(H_{ij}^1, H_{jt}^2, H_{jt}^3, H_i^4\) is obtained as \(H(\hat{h}_{ij}^1)\), \(H(\hat{h}_{jt}^2)\), \(H(\hat{h}_{jt}^3)\), and \(H(\hat{h}_i^4)\), respectively; similarly, each matrix \(\tilde{H}_{ij}^1, \tilde{H}_{jt}^2, \tilde{H}_{jt}^3, \tilde{H}_i^4\) is obtained as \(\tilde{H}(\hat{h}_{ij}^1)\), \(\tilde{H}(\hat{h}_{jt}^2)\), \(\tilde{H}(\hat{h}_{jt}^3)\), and \(\tilde{H}(\hat{h}_i^4)\), respectively.

For example, for the constraint
\[
\sum_{t\in\T} s_i(t)\,p_j(t)-e_i x_{ij}=\phi_{ij},
\]
the corresponding coefficient vector \(\hat{h}_{ij}^1\in\mathbb{R}^N\) is defined by
\[
(\hat{h}_{ij}^1)_{\mathrm{pos}(x_{ij})}=-e_i,\qquad
(\hat{h}_{ij}^1)_{\mathrm{pos}(p_j(t))}=s_i(t)\ \ \forall t\in\T,\qquad
(\hat{h}_{ij}^1)_{\mathrm{pos}(\phi_{ij})}=-1,
\]
and all remaining entries are zero. Equivalently,
\[
\hat{h}_{ij}^1
=
-e_i\,\mathbf{e}_{\mathrm{pos}(x_{ij})}
+\sum_{t\in\T}s_i(t)\,\mathbf{e}_{\mathrm{pos}(p_j(t))}
-\mathbf{e}_{\mathrm{pos}(\phi_{ij})},
\]
where \(\mathbf{e}_k\) denotes the \(k\)-th standard basis vector in \(\mathbb{R}^N\).